\documentclass[12pt]{amsart}
\usepackage{graphicx} 
\usepackage[top=3cm, bottom=5cm, left=3cm, right=3cm]{geometry}

\usepackage[english]{babel}
\usepackage{amsmath}
\usepackage{amsthm}
\usepackage{amsfonts}
\usepackage{graphicx}
\usepackage{tikz-cd}
\usepackage[colorlinks=true, allcolors=black]{hyperref}
\usepackage{comment}
\usepackage{mathtools}
\usepackage[normalem]{ulem}
\usepackage{amsaddr}

\newtheorem{theorem}{Theorem}[section]

\newtheorem{prop}[theorem]{Proposition}

\newtheorem{cor}[theorem]{Corollary}

\theoremstyle{definition}
\newtheorem{definition}[theorem]{Definition}

\newtheorem{example}[theorem]{Example}

\numberwithin{equation}{section}

\newcommand{\bZ}{{\mathbb Z}}
\newcommand{\bR}{\mathbb R}
\newcommand{\bC}{{\mathbb C}}

\newcommand{\ra}{\rightarrow}

\begin{document}
\raggedbottom
\title{Adapted complex polarizations on solvable Lie groups of dimension 3}
\author{Balázs Forman}
\address{ELTE - Eötvös L. Univ., Dept. of Anal.,\\
 P\'azm\'any P\'eter s\'et\'any 1/C, 1117, Budapest, Hungary}
\email{formanbalazsattila@gmail.com}


\begin{abstract}

In this paper we investigate adapted complex polarizations associated with left-invariant Koszul connections on real Lie groups. Adapted complex polarizations arise as a natural generalization of a complex structure on a manifold. We investigate the existence of such polarizations on the tangent bundle of certain Lie groups with left-invariant connections. We use a practical criterion for the existence of global adapted complex polarizations in terms of the $i\mathbb{R}$-completeness of the geodesic flow of the connection. In the Lie case, the Euler–Arnold vector field determines the behaviour of the geodesic flow. By combining what we know about the Euler-Arnold vector field of certain Lie groups and the $i\bR$-completeness criterion, we provide several examples of left-invariant ac-polarizations defined on the tangent bundle of Lie groups. Along with pseudo-Riemannian metric connections, we discuss biinvariant connections in several cases. Our focus is on solvable Lie groups, because non-solvable Lie groups are better understood. The main result of this paper is that we demonstrate several not locally symmetric but $\mathcal P$-complete connections.
\end{abstract}

\maketitle

\section{Introduction}

Adapted complex structures on tangent bundles were introduced by Lempert and Szőke in their study of the homogeneous complex Monge-Ampère equation and have since become an important tool in the complexification of Riemannian and more general Koszul manifolds \cite{LSz91}, \cite{Sz04}, \cite{L19}. Their work and subsequent developments by Halverscheid, Iannuzzi \cite{HI03}, \cite{HI09} and others revealed deep connections between geodesic geometry, complexification, and involutive structures on tangent bundles. In particular, the existence of a global adapted complex structure is closely related to the holomorphic extension properties of geodesics.

A natural generalization of the complexification of Koszul manifolds is obtained by replacing complex structures with involutive structures and polarizations. Following Lempert, one may consider adapted complex polarizations (ac-polarizations), which are involutive complex subbundles of the complexified tangent bundle that are compatible with the geodesic flow. This notion is particularly well suited to left-invariant connections on Lie groups.

For a left-invariant Koszul connection on a Lie group, geodesics are governed by the Euler-Arnold equation, a fundamental observation going back to Arnold \cite{A66}. Thus, questions about adapted complex polarizations can be translated into questions about the holomorphic extendability of solutions of quadratic homogeneous differential equations on the Lie algebra. The key result used throughout this paper is a criterion that asserts that if the Euler-Arnold vector field is $i\mathbb{R}$-complete in the complexified Lie algebra, then the universal complexification of the Lie group induces a global ac-polarization on its tangent bundle \cite{FSz26}. Consequently, the problem of constructing adapted complex polarizations becomes a problem in complex ODE theory.

The main goal of this paper is to apply this principle to left-invariant pseudo-Riemannian metrics and biinvariant connections on solvable Lie groups of dimension three. To do this, recall that the classification of Lie groups in dimension 3 is called the Bianchi classification. \cite{GKKL15} Many of the groups of dimension three are already well-understood.  Type I groups are Abelian, which are trivial from our point of view, since metric connections vanish and any left-invariant connection is biinvariant. Type II groups are isomorphic to the Heisenberg group, which was treated (along with all other 2-step nilpotent groups) in \cite{FSz26}. Type VIII and Type IX Lie groups are simple and were discussed in \cite{FSz26} and \cite{HI09}, respectively. All other Lie groups of dimension three are solvable and are discussed in detail in this article.

Our first set of results concerns the motion groups $E(2)$ and $E(1,1)$. Using Guediri’s \cite{G96} classification of left-invariant Lorentz metrics and the corresponding Euler-Arnold equations, we determine exactly which metrics have $i\mathbb{R}$-complete Euler--Arnold flows. In each case, the analysis reduces to Riccati-type or quartic differential equations, allowing a complete description of the polarization-complete metrics. 

We then study biinvariant connections. A simple family of such connections exists on Lie groups whose commutator ideal has codimension one, and we show that these connections are always polarization complete. We further analyze solvable non-unimodular Lie groups in dimension three, obtaining a description of all torsion-free biinvariant connections in the degenerate case and exhibiting examples that fail to be $i\mathbb{R}$-complete.

Finally, we consider the broader class of general $\mathcal F$-type Lie groups. We prove that these groups admit no biinvariant pseudo-Riemannian metrics, classify all torsion-free biinvariant connections on them, establish their polarization completeness, and compute explicit formulas for their curvature tensor and its covariant derivative.

The paper is organized as follows. Section 2 recalls the necessary background on adapted complex structures, polarizations, complexifications, and Euler--Arnold equations. Section 3 develops the analytic tools used to test $i\mathbb{R}$-completeness of polynomial vector fields. Sections 4 and 5 classify polarization-complete left-invariant pseudo-Riemannian metrics on $E(2)$ and $E(1,1)$, respectively. Section 6 discusses biinvariant connections, while Sections 7 and 8 treat solvable non-unimodular Lie groups and general $\mathcal F$-type Lie groups.

\noindent \textbf{Acknowledgement.} This work was done while the author participated in the PhD program of the Institute of Mathematics at Eötvös Lóránd University. He would like to thank his supervisor Róbert Szőke and the Institute for their hospitality and very inspiring scientific atmosphere. Also, he received a scholarship from the Hungarian state and one from the Erzsébetváros district of Budapest, for which he is grateful.

\section{Polarizations}

All definitions, theorems, and propositions in this section are from the joint work with Róbert Szőke in \cite{FSz26}.

First, we look at the general setup. In case we have a Koszul manifold $(M,\nabla)$ we can define a map $\Omega_v$ associated to the geodesic with initial velocity $v$ going from a complex domain to $TM$ as follows.

\begin{definition}
    Let $M$ be a smooth manifold together with a  Koszul connection $\nabla$ on $TM$. 
For  $I=(a,b)$, $a<0<b$ let
\begin{equation}\label{E:ueps}
\bC I=\{z=\sigma+i\tau\in\mathbb C \mid \text{Re }z\in I\}.
\end{equation}
Let $v\in TM$ and 
$\gamma_v:I\ra M$ be the geodesic
with initial condition $\dot\gamma_v(0)=v$. Define
\begin{equation}\label{F:Omegav}
\begin{aligned}
\Omega_v:\bC I&\ra TM\\
\sigma+i\tau&\mapsto \tau\dot{\gamma}(\sigma)
\end{aligned}
\end{equation}
\end{definition}

The image of such a map is a two-dimensional submanifold of $TM$, and we want to construct a complex structure on $TM$ such that each of these leaves is a complex submanifold.

\begin{definition} 
 Let  $N\subset TM$ be an open, fiberwise $*$-domain (i.e. $v\in N$ and $0\le s\le1$ implies $sv\in N$)  with $M\subset N$. An {\it adapted complex structure} (ac-structure for short) on $N$ means a complex manifold structure so that for every
$v\in TM$, the map $\Omega_v$ is holomorphic on $\Omega_v^{-1}(N)$.
$\nabla$ is called {\it entire}, if an ac-structure exists on $TM$.
A  pseudo-Riemannian manifold $(M,g)$    is called {\it entire} if the Lévi-Civita connection of $g$ is entire.\cite[Definition 2.1.]{FSz26} 
\end{definition}

Constructing ac-structures is immensely difficult; therefore, we define a weaker structure, the so-called ac-polarizations. 

\begin{definition}
  Let $N$ be a smooth manifold. An involutive structure  on $N$
 is a smooth, involutive, complex subbundle $P$ of $\mathbb CTN$.  Involutivity means that the Lie bracket of local smooth sections of  $P$  is again
 a section of $P$. An involutive structure $P$ 
 is called a polarization   if $\dim N$
 is even and
 rank$P=\frac1{2}\dim N$.  \cite[Definition 2.2.]{FSz26}
\end{definition}

Since $N$ is not a complex manifold, it does not have complex submanifolds. Instead, what we can require is that the pushforward of $T^{1,0}\bC I$ be in $P$. This is a direct generalization, since in the case of $P=T^{1,0}N$ this is equivalent to the requirement $\Omega_v$ be holomorphic on $\Omega^{-1}(N)$ for all $v\in TM$.

 \begin{definition}
 Let   $(M,\nabla)$ be a Koszul manifold and  $N\subset TM$ an open, fiberwise $*$-domain with $M\subset N$.
 An  involutive complex subbundle $P\subset\bC TN$ is called an {\it adapted involutive  structure} (ai-structure)  if 
 ${\Omega_v}_*(T^{1,0}(\Omega_v^{-1}(N))\subset P$ for every $v\in TM$. If the ai-structure $P$ is a polarization, then we call it an  {\it adapted complex polarization} (or ac-polarization).    
 $\nabla$ is said to be {\it polarization complete} or simply $\mathcal P$-{\it complete} if its ac-polarization  exists on $TM$.   
  A pseudo-Riemannian manifold $(M,g)$ is said to be $\mathcal P$-{\it complete}  if its Levi-Civita connection is $\mathcal P$-{\it complete}.\cite[Definition 2.3.]{FSz26}
\end{definition}

The notion of ac-polarization is more general than that of an ac-structure. An ac-polarization $P$ is an ac-structure if and only if $P\cap\overline P=0$. For constructing an ac-polarization, we use the method of Lempert \cite{L19}.

We first define the domain in which our ac-polarization lives.

 \begin{definition}
     Let $(M,\nabla)$ be a Koszul manifold, $Z$ be a complex manifold and $f: M \ra Z$ a smooth map. For  $v\in T_pM$ let $\gamma_v:I_p\ra M$ be the $\nabla$-geodesic with initial condition $\dot\gamma_v(0)=v$  and $\chi_v:=f\circ\gamma_v$. Define
 \begin{equation}\label{E:Ni} 
N_i(f,Z):=\{v\in TM : \chi_v\text{ extends holomorpically to a neighborhood of } [0,i]\}
 \end{equation}

 For  $v\in T_pM$, let  $\Omega_v$ be   the map from \eqref{F:Omegav} and  $\chi^\bC_v $ the holomorphic extension of $\chi_v$.
 \end{definition}

 This motivates the following definition.

\begin{definition}\label{D:polarmap}
  Let $(M,\nabla)$ be a Koszul manifold, $Z$ a complex manifold and $f:M\ra Z$ a smooth map. The corresponding {\it polar map} is
  \begin{equation}\label{E:polarmap}
      \psi_f:N_i(f,Z)\ra Z,\qquad v\mapsto
      \chi^\bC_v(i).
  \end{equation}\cite[Definition 2.5.]{FSz26}
\end{definition}

To use both the connection of $M$ and the complex structure of $Z$ we need to introduce the notion of holomorphic Koszul connections.

\begin{definition}
    Let $Z$ be a complex manifold, $T^{1,0}Z$ its holomorphic tangent bundle and $\pi:T^{1,0}Z\ra Z$ the bundle projection.  A holomorphic Koszul connection $\mathfrak D$ on $T^{1,0}Z$, is a $\bC$-linear map (of sheaves) $\mathfrak D:T^{1,0}Z\ra\Omega_Z\otimes T^{1,0}Z$ which satisfies the Leibniz rule
 \begin{equation}
     \mathfrak D(f\cdot s)=\partial f\otimes s+f\cdot\mathfrak D s,
 \end{equation}
 where $\Omega_Z$ denotes the holomorphic cotangent sheaf, $f$ any local holomorphic function on $Z$ and $s$ any local holomorphic vector field on $Z$. For a $w\in T^{1,0}Z$, let $\Gamma_w$ be the holomorphic $\mathfrak D$-geodesic defined in some open neighborhood $D\subset \bC$ of $0$ with initial condition $\dot\Gamma^{1,0}_w(0)=w$ and define $\chi_w:D\ra T^{1,0}Z$, with $\zeta\mapsto\dot\Gamma^{1,0}_w(\zeta)$ and finally let $\mathcal V(w):=\dot\chi^{1,0}_w(0)$.  Then $\mathcal V :X\ra T^{1,0}X$ is a holomorphic (1,0) vector field on the complex manifold $X:=T^{1,0}Z$, the total space of the holomorphic tangent bundle.\end{definition}

If we look at a map $f:M\ra Z$ between a real and a complex Koszul manifold, there is no natural way to connect $\nabla$ to  $\mathfrak D$. We need an additional step, which is to associate a real Koszul connection $\hat\nabla$ to $\mathfrak{D}$.

 \begin{definition}
     The connection $\mathfrak D$ induces an ordinary  Koszul connection $\widehat\nabla$ on $TZ$ as follows. The $\bC$-bundle isomorphism  $\xi:TZ\ra T^{1,0}Z, u\mapsto u^{1,0}$ yields a complex manifold structure on the smooth manifold $TZ$ by pulling back the complex manifold structure of the total space $T^{1,0}Z$. The pulled-back vector field $\xi^*(V)$ is a real vector field on the total space  $TZ$, the geodesic spray of $\widehat\nabla$. Looking at  this from a different angle:  $\mathfrak D+\bar{\partial}$ will be an ordinary Koszul connection on the complex vector bundle $T^{1,0}Z$ that yields via the  $\bC$-bundle isomorphism $\xi$ the Koszul connection $\widehat \nabla$ on $TZ$.
 \end{definition}

 One can show that for  any local holomorphic (1,0) vector field $u$ in $Z$, 
$\widehat\nabla (u+\bar u)=\mathfrak Du+\overline{\mathfrak Du}$ and in fact   this formula could be used to define $\widehat \nabla$. 
What is important for us is the following property of $\widehat\nabla$ (c.f. \cite{LB80, LB83}):  
\begin{equation}\label{E:geod}
\widehat\nabla\text{-geodesics are precisely the restrictions of holomorphic } \mathfrak D\text{-geodesics to } \mathbb R.
\end{equation}

With the help of $\hat \nabla$ we can now define what we mean by the complexification of a Koszul manifold.

 \begin{definition}\label{D:complexification}
 Let $M$ be a real-analytic manifold with a real-analytic  Koszul connection $\nabla$. A {\it complexification} of $(M,\nabla)$ is a triple $(Z,\mathfrak D, \iota)$, where $Z$ is a complex manifold, $\mathfrak D$  a holomorphic Koszul connection on $T^{1,0}Z$, $\widehat\nabla$ the corresponding real Koszul connection on $TZ$ and $\iota:M\ra Z$ a maximally totally real, real-analytic immersion with  $\nabla=\iota^*(\widehat\nabla)$.  \cite[Definition 5.1.]{FSz26}
\end{definition}

 \begin{definition}\label{D:icompl-complex}
    The   complexification $(Z,\mathfrak D,\iota)$ of $(M,\nabla)$ is called {\it $\bC$-complete}, if for every $v\in TM$, $\chi_v$ extends holomorphically to $\bC$; it is called {\it $i\bR$-complete} if for each $v$, $\chi_v$ extends holomorphically to a neighborhood of $i\bR$ or equivalently if $N_i(\iota,Z)=TM$.\cite[Definition 5.3.]{FSz26}
\end{definition}

\subsection{Lie groups}

The main focus of this paper is a special class of manifolds, Lie groups. Lie groups have natural algebraic complexifications, and we can use the group structure to investigate not just any connection on them, but left-invariant ones.

Let $G$ be a real Lie group with unit element $e$ and Lie algebra  $\mathfrak g=T_eG$.
 Left invariant Koszul connections $\nabla$ on $TG$ and   $\mathbb R$-bilinear maps 
$\alpha:\mathfrak g\times\mathfrak g\ra\mathfrak g$ 
bijectively correspond to each other 
 by  
$\alpha(x,y):=(\nabla_xy)_e$, where
$x,  y\in T_eG$  (the same letters denote the left invariant vector field extensions as well).

\begin{definition}The vector field $\mathcal E$ on $\mathfrak g$, defined by the quadratic homogeneous polynomial map
 $\mathfrak g\ra \mathfrak g$, $x\mapsto -\alpha(x,x)$ 
 is called the {\it Euler-Arnold vector field}  of $\nabla$. 
\end{definition}

The Euler-Arnold vector field is essential for our work, since, due to the following Proposition, the behavior of left-invariant geodesics is governed by it.

\begin{prop}\label{P:leftinvconn}
    Let  $G$ be a real Lie group with Lie algebra $\mathfrak g$. Let $\nabla$ be a left invariant Koszul connection on $TG$ and $\alpha:\mathfrak g\times\mathfrak g\ra\mathfrak g$ the 
 corresponding bilinear map.
Let $\gamma:(a,b)\ra G$ be a smooth curve and define $x:(a,b)\ra\mathfrak g$ by 
$x(t):=(L_{\gamma(t)}^{-1})_*(\dot\gamma(t))$. Then $\gamma$ is a $\nabla$-geodesic iff
$$\dot x+\alpha(x,x)=0,$$
i.e. iff $x$ is an integral curve of the Euler-Arnold vector field.\cite[Proposition 6.1.]{FSz26}
\end{prop}

It is important to point out that a left-invariant connection $\nabla$ on $TG$ is $\bC$-complete (resp. $i\bR$-complete) exactly when the corresponding Euler-Arnold vector field is $\bC$-complete (resp. $i\bR$-complete).

In the metric case, the Koszul formula describes the Levi-Civita connection in terms of the metric. If our connection is both metric and left-invariant, we arrive at the following conclusion.

\begin{cor}\label{C:EAequation} Let $G$
be a real Lie group and $q$ a left-invariant pseudo-Riemannian metric on $TG$. Then the Levi-Civita connection $\nabla$ of $q$ is
 \begin{equation}\label{E:LeviCivita}
  \nabla_XY=\frac1{2}\{[X,Y]-ad^*_XY-ad^*_YX\}, \quad X, Y\in\mathfrak g.   
 \end{equation}      
       Let $\gamma:(a,b)\ra G$ be a smooth curve and   $x:(a,b)\ra\mathfrak g$ defined by 
$x(t):=(L_{\gamma(t)}^{-1})_*(\dot\gamma(t))$. Then
     $\gamma$ is a geodesic iff
$$
\dot x=ad^*_xx.
$$ 
 The Euler-Arnold vector field is $\mathcal E(x)=-\alpha(x,x)=ad^*_xx$.\cite[Corollary 6.2.]{FSz26}
\end{cor}

The following theorem will be our reference point throughout this article.

\begin{theorem}\label{T:nabla-i-complete} Let $G$ be a connected, real Lie group with Lie algebra $\mathfrak g$, $\nabla$ a left invariant Koszul connection  on $TG$ and $\eta:G\ra G_\bC$ the universal complexification. Suppose the Euler-Arnold vector field  $\mathcal E$    is $i\bR$-complete in $\mathfrak g\otimes\bC$.(See Definition \ref{D:iR complete}).
Then $N_i(\eta,G_\bC)=TG$ and with the polar map $\psi_\eta$, associated to $\eta$,
$P:=(\psi_\eta)^{-1}_*(T^{1,0}G_\bC)$ is a $G$-invariant ac-polarization on $TG$.
In particular
$\nabla$ is $\mathcal P$-complete.\cite[Theorem 7.5.]{FSz26}
\end{theorem}

\section{\texorpdfstring{$i\bR$}-complete vector fields}

In this section, we investigate several vector fields in terms of $i\bR$-completeness. Our main goal is to collect tools to determine the $i\bR$-completeness of the Euler-Arnold vector fields.

\begin{definition}\label{D:iR complete}
 A smooth real vector field $\mathcal W$ on $\bR^n$ is called   {\it $\bC$-complete} (resp. 
      {\it $i\bR$-complete}) if every  trajectory extends  to
a holomorphic map  $\bC\ra\bC^n$ (resp.  
$U\ra\bC^n$, where $U$ is
some open neighborhood of $i\bR$, that can depend on the trajectory).\cite[Definition 7.4.]{FSz26}
 \end{definition}

\begin{theorem}\label{thm:x'^2}
    Let $c_1,c_2\in\bC$, then the following ODE:

    \begin{equation}\label{eq: x'^2}
     \dot{x}^2(t)=c_1 x^2(t)+c_2 x^4(t),
    \end{equation}

    has solutions as follows:
    \begin{enumerate}
        \item If $c_1=0=c_2$, then $x(t)$ is constant.
        \item If $c_1=0$, $c_2\neq 0,$ then $$x(t)=\frac{x(0)}{1\pm x(0)\sqrt{c_2}t}$$
        \item If $c_1\neq 0$, $c_2=0$ then $$x(t)=x(0)e^{\pm\sqrt{c_1}t}$$
        \item If $c_1\neq 0\neq c_2,$ then $$x(t)=\sqrt{\frac{c_1}{c_2}}\frac{2c_3e^{\sqrt{c_1}t}}{1-c_3^2 e^{2\sqrt{c_1}t}},$$
    \end{enumerate}

    where $c_3=\frac{-\sqrt{\frac{c_1}{c_2}}\pm\sqrt{x(0)^2+\frac{c_1}{c_2}}}{x(0)}\in\bC$, if $x(0)\neq 0$ and $c_3=0$, if $x(0)=0$. The square root sign can take both possible complex values in each case. Furthermore, under the assumption that $c_1,c_2,c_3\in\bR$, the solution is $i\bR$-complete iff $c_2\geq 0$.
\end{theorem}

\begin{proof}

    The first three cases are elementary and show thatentire solution exists only if $c_2=0$.

    In the case when $c_1\neq0\neq c_2$, note that since equation \eqref{eq: x'^2} is separable, it suffices to solve 

    \begin{equation}\label{eq:integral}
        \int \frac{\dot{x}(t)dt}{x(t)\sqrt{1+\frac{c_2}{c_1} x^2}}=\sqrt{c_1}t+c.
    \end{equation}

    We now consider the following substitution:

    \begin{equation}\label{eq:substitution}\begin{split}
        x(t)&=\sqrt{\frac{c_1}{c_2}}\frac{v(t)-\frac{1}{v(t)}}{2} \\
        \dot{x}(t)&=\dot{v}(t)\sqrt{\frac{c_1}{c_2}}\frac{1+\frac{1}{v^2(t)}}{2}=\dot{v}(t)\sqrt{\frac{c_1}{c_2}}\frac{v(t)+\frac{1}{v(t)}}{2v(t)} \\
        1+\frac{c_2}{c_1}x^2(t)&=1+\left(\frac{v(t)-\frac{1}{v(t)}}{2}\right)^2=\left(\frac{v(t)+\frac{1}{v(t)}}{2}\right)^2
    \end{split}\end{equation}

    The left hand side of equation \eqref{eq:integral} becomes

    \begin{equation}\label{eq:integral sub}
        \begin{split}
        \int \frac{\dot{x}(t)dt}{x(t)\sqrt{1+\frac{c_2}{c_1} x^2(t)}}&=\int \frac{\dot{v}(t)\sqrt{\frac{c_1}{c_2}}\frac{v(t)+\frac{1}{v(t)}}{2v(t)}dt}{\sqrt{\frac{c_1}{c_2}}\frac{v(t)-\frac{1}{v(t)}}{2}\frac{v(t)+\frac{1}{v(t)}}{2}}\\
        &=\int \frac{2\dot{v}(t)dt}{v^2(t)-1}=\int\frac{\dot{v}(t)dt}{v(t)-1}-\int\frac{\dot{v}(t)dt}{v(t)+1}\\
        &=\log\left(\frac{v(t)-1}{v(t)+1}\right)+C
        \end{split}
    \end{equation}

    Using equation \eqref{eq:integral sub} and exponentiating both sides of equation \eqref{eq:integral} we obtain

    \begin{equation}
        \frac{v(t)-1}{v(t)+1}=e^{\sqrt{c_1}t+c}.
    \end{equation}

    By rearranging, we get
    
    \begin{equation}
        v(t)=\frac{1+e^{\sqrt{c_1}t+c}}{1-e^{\sqrt{c_1}t+c}}.
    \end{equation}     

    Using substitution \eqref{eq:substitution} 
    
    \begin{equation}\begin{split}\label{eq:x'^2 x}
        x(t)&=\sqrt{\frac{c_1}{c_2}}\frac{\frac{1+e^{\sqrt{c_1}t+c}}{1-e^{\sqrt{c_1}t+c}}-\frac{1-e^{\sqrt{c_1}t+c}}{1+e^{\sqrt{c_1}t+c}}}{2}\\
        &=\sqrt{\frac{c_1}{c_2}}\frac{2e^{\sqrt{c_1}t+c}}{1-e^{2(\sqrt{c_1}t+c)}}\\
        &=\sqrt{\frac{c_1}{c_2}}\frac{2c_3e^{\sqrt{c_1}t}}{1-c_3^2 e^{2\sqrt{c_1}t}}
        \\
        &=\sqrt{\frac{c_1}{c_2}}\frac{1}{\sinh({\sqrt{c_1}t+c})},
    \end{split}\end{equation}
    where $c_3\in\bC$, and $c_3^2\neq 1$. This function has a pole at $t_0$ exactly when

    \begin{equation}
        \sqrt{c_1}t_0+c=ik\pi,
    \end{equation}

    where $k\in\bZ$. So

    \begin{equation}\label{eq:poles}
        t_0=\frac{ik\pi-c}{\sqrt{c_1}}
    \end{equation}
    
    From now on, suppose that $c_1,c_2\in\bR$. For $t_0$ in \eqref{eq:poles} to be purely imaginary, either $c\in\bR$, and $c_1<0$, or $c\in i\bR$, and $c_1>0$. 

    The value of $c$ depends on the initial condition in the following way according to equation \eqref{eq:x'^2 x} 
    
    \begin{equation}
        \sinh(c)=\sqrt\frac{c_1}{c_2x(0)^2}.
    \end{equation}

    The right hand side is either purely imaginary, or real, since by our assumption $c_1,c_2,x(0)\in\bR$. Recall that $\sinh(c)=\sinh(a)\cos(b)+i\cosh(a)\sin(b)$, where $c=a+ib\in\bC$. It is easy to see from this that 

    \begin{equation}
        \begin{split}
            \sinh(c)\in i\bR,\  |\sinh(c)|\geq1&\iff b=\frac{\pi}{2}+k\pi,\text{ where } k\in\bZ,\\
            \sinh(c)\in i\bR,\  |\sinh(c)|\leq 1&\iff a=0 \\
            \sinh(c)\in \bR\setminus\{0\},&\iff b=k\pi,\text{ where } k\in\bZ,\text{ and } a\neq 0
        \end{split}
    \end{equation}

    If $c_1,c_2>0$, then $b=k\pi$, $a\neq 0$. Hence $t_0$ takes no purely imaginary values.

    If $c_1>0$, $c_2<0$, then $c\in i\bR$ if $|\sqrt\frac{c_1}{c_2x(0)^2}|\leq 1$, which holds for $x(0)^2\geq |\frac{c_1}{c_2}|$. Therefore $x(t)$ will have an imaginary pole if $x(0)$ is big enough.

    If $c_1<0$, $c_2>0$, then $c\not\in\bR$, therefore Equation \eqref{eq:poles} takes no purely imaginary values.

    If $c_1,c_2<0$ then $c\in \bR+ki\pi$, where $k\in\mathbb Z$, therefore Equation \eqref{eq:poles} has purely imaginary solutions.\end{proof}

    The following two technical Propositions were proven in \cite{FSz26}, but since we use them extensively in this paper as well, we present them here.

\begin{prop}\label{P:t^2+a}  
Let  $A,B\in\bR$. Consider the following vector fields on $\bR$.
\begin{enumerate}
    \item[a)] $\mathcal{W}(y):=(Ay^2+B)\partial_y$  is  $i\bR$-complete  iff   $A=0$ or $\frac{B}{A}\ge0$. When  $\frac{B}{A}<0$, $\mathcal W$ is  $i\bR$-complete at every point except at $p=0$.
    \item[b)] $\mathcal{W}(y):=(Ay^2+By)\partial_y$ is $i\bR$-complete  iff $A=0$.When $A\not=0$, $\mathcal W$ is 
 $i\bR$-complete at every pont  except  at 
 $p=-\frac{B}{2A}\in \bR$.
\end{enumerate}\cite[Proposition 12.5.]{FSz26}
\end{prop}

\begin{prop}\label{P:l-L-Omega} 
Let $n$ be a positive integer.
\begin{enumerate}
    \item[a)] For  an $l\in {(\bR^n)}^*$
 define the  map $\varphi:\bR^n\ra\bR^n$ by
$\varphi(x):=l(x)x$ and let $\mathcal W$ be the corresponding
   vector field on $\bR^n$.
Then $\mathcal W$ is $i\bR$-complete.
More generally we have.
\item[b)] Suppose   $\bR^n=W\oplus V$ is a decomposition  to vector subspaces. Let $A:W\times W\ra W$ and  $B:\bR^n\times\bR^n\ra V$ be  symmetric bilinear maps and $\left.B\right|_{V\times V}\equiv0$. Define $\psi(u):=A(u,u), u\in W$
and $\varphi(x):=A(x_w,x_w)\oplus B(x,x)$.
Denote by  $\mathcal V_w$ the corresponding vector field to $\psi$
 on $W$  
 (resp. of $\varphi$  on $\bR^n$ by $\mathcal W$). Assume that $\mathcal V_w$ is $i\bR$-complete.
 Then $\mathcal V$ is $i\bR$-complete. When $A=0$, 
$\mathcal V$ will be even  $\bC$-complete. 
\end{enumerate}\cite[Proposition 12.7.]{FSz26}
\end{prop}

\section{The group E(2)}

Throughout the next two chapters, we will follow the notation and framework of Guediri \cite{G96}, who classified the left-invariant Lorentzian metrics on $E(2)$ and $E(1,1)$. The corresponding left-invariant Riemannian metrics were later classified by Inoguchi and Van der Veken \cite{IV07}. Notably, the first case of Guediri's Lorentzian classification directly encompasses the Riemannian case, up to the signs of the parameters $\lambda_i$. Consequently, these two cases can be treated simultaneously.

Let $E(2)$ denote the group of isometries of the Euclidean plane, which can be represented as

\begin{equation}
    E(2)=\left\{ \begin{pmatrix}
    \cos(t) & -\sin(t) & a \\
    \sin(t) & \cos(t) & b \\
    0 & 0 & 1
\end{pmatrix}\in GL(3,\bR)\Bigg|a,b,t\in\bR\right\}.
\end{equation}

Denote $[\mathfrak{e}(2),\mathfrak{e}(2)]$ by $V$. Let $e_1,e_2$ be a basis of $V$ and $e_0\in\mathfrak{e}(2)\setminus V$ such that $[e_0,e_1]=e_2$ and $[e_0,e_2]=-e_1$.
Let us define a bilinear form $(.,.)$ on $V$ such that

\begin{equation}
    (x,y)=x_1y_1+x_2y_2,
\end{equation}

where $x=x_1e_1+x_2e_2$ and $y=y_1e_1+y_2e_2$.

\begin{theorem}\label{thm:e2 left-invariant}
    Let $q$ be a left-invariant pseudo-Riemannian metric on $E(2)$. Then the Euler-Arnold vector field is $i\bR$-complete iff
    \begin{equation}\label{eq:e2 i-complete metric}
        q=\lambda_0(e_0^*)^2+\lambda((e_1^*)^2+(e_2^*)^2).
    \end{equation}
    In this case the Euler-Arnold vector field is $\bC$-complete.
\end{theorem}

\begin{proof}

According to \cite{G96} a pseudo-Riemannian metric $q$ on $E(2)$ is always complete in the real sense.

\textit{Case 1.} $q|_V$ is definite. Then for a trajectory $x(t)=x_0(t)e_0+x_1(t)e_1+x_2(t)e_2$, the Euler-Arnold equations are 

\begin{equation}\begin{split}\label{eq:e2 nondeg EA}
    \dot{x}_0(t)&=\frac{\lambda_2-\lambda_1}{\lambda_0}x_1(t)x_2(t) \\
    \dot{x}_1(t)&=-\frac{\lambda_2}{\lambda_1}x_0(t)x_2(t) \\
    \dot{x}_2(t)&=\frac{\lambda_1}{\lambda_2}x_0(t)x_1(t).
\end{split}\end{equation}

It is easy to see that if $\lambda_1=\lambda_2$ then the trajectories are $\bC$-complete. Since then $x_0(t)$ remains constant and $x_1(t)$ and $x_2(t)$ form a harmonic oscillator.

In the following, we will suppose that $\lambda_1\neq\lambda_2$. To disprove $i\bR$-completeness, we will construct a trajectory which is not $i\bR$-complete using first integrals and carefully choosing the initial values.

This system of equations has 2 first integrals:

\begin{equation}\begin{split}\label{eq:e2 nondeg first integrals}
    \lambda_0x_0^2(t)+\lambda_1x_1^2(t)+\lambda_2x_2^2(t)=&e \\
    \lambda_1^2 x_1^2(t)+\lambda_2^2 x_2^2(t)=&m\ 
\end{split}\end{equation}

Since $\lambda_1,\lambda_2, \lambda_0\neq 0$, we can express $x_2(t)$ and $x_1(t)$ from $x_0(t)$ in the following way:

\begin{equation}\begin{split}\label{eq:e2 nondeg x_2 x_0}
    \lambda_1e-m&=\lambda_1\lambda_0x_0^2(t)+(\lambda_1\lambda_2-\lambda_2^2)x_2^2(t)\\
    x_2^2(t)&=\frac{\lambda_1e-m-\lambda_1\lambda_0x_0^2(t)}{\lambda_2(\lambda_1-\lambda_2)}
\end{split}\end{equation}

and

\begin{equation}\begin{split}\label{eq:e2 nondeg x_1 x_0}
    \lambda_2e-m&=\lambda_2\lambda_0x_0^2(t)+(\lambda_2\lambda_1-\lambda_1^2)x_1^2(t)\\
    x_1^2(t)&=\frac{\lambda_2e-m-\lambda_2\lambda_0x_0^2(t)}{\lambda_1(\lambda_2-\lambda_1)}.
\end{split}\end{equation}

Now, to proceed, let's calculate $\dot{x}_0^2(t)$ by substituting equation \eqref{eq:e2 nondeg x_1 x_0} and \eqref{eq:e2 nondeg x_2 x_0} into \eqref{eq:e2 nondeg EA} to yield

\begin{equation}\begin{split}
    \dot{x}_0^2(t)&=\left( \frac{\lambda_2-\lambda_1}{\lambda_0}\right)^2\frac{\lambda_2e-m-\lambda_2\lambda_0x_0^2(t)}{\lambda_1(\lambda_2-\lambda_1)}\frac{\lambda_1e-m-\lambda_1\lambda_0x_0^2(t)}{\lambda_2(\lambda_1-\lambda_2)}=\\
    &=-x_0^4(t)+\frac{(\lambda_2+\lambda_1)e-2m}{\lambda_0}x_0^2(t)-\frac{(\lambda_1e-m)(\lambda_2e-m)}{\lambda_0^2\lambda_1\lambda_2}.
\end{split}\end{equation}

If $(\lambda_1e-m)(\lambda_2e-m)=0$, then we can use Theorem \ref{thm:x'^2} and conclude that a trajectory belonging to the given initial condition is not $i\bR$-complete.

We need to prove that there exist initial conditions such that $\lambda_1e=m$ or $\lambda_2e=m$.

With the choice of

\begin{equation}
    x_0(0)=x_{j}(0)=0,\quad x_{3-j}(0)\neq 0, \quad \text{for } j=1,2
\end{equation}

$(\lambda_1e-m)(\lambda_2e-m)=0$ is satisfied so the corresponding trajectory, and geodesic, is not $i\bR$-complete.

\textit{Case 2.} If $q|_V$ is indefinite, we can form a basis, for which the Euler-Arnold equations are:

\begin{equation}\begin{split}
    \dot{x}_0(t)&=x_0(t)x_2(t) \\
    \dot{x}_1(t)&=-(x_0(t)+x_1(t))x_2(t) \\
    \dot{x}_2(t)&=-x_0^2(t),
\end{split}\end{equation}

where $x(t)=x_0(t)e_0+x_1(t)e_1+x_2(t)e_2$. 

In this case the two first integrals are:

\begin{equation}\begin{split}
    x_2^2(t)-2x_0(t)x_1(t)&=e \\
    x_0^2(t)+x_2^2(t)&=m
\end{split}\end{equation}

From these two expressions, we can easily deduce that

\begin{equation}
    \dot{x}_2(t)=x_2^2(t)-m.
\end{equation}

Thus $x_2(t)$ is not $i\bR$-complete when $m>0$ and $x_2(0)=0$ according to Proposition \ref{P:t^2+a}. The trajectory belonging to an initial condition choice of $x_0(0)\neq 0$ and $x_2(0)=0$ is not $i\bR$-complete, and therefore neither is this family of pseudo-Riemannian metrics.\end{proof}

\begin{theorem}
    From Theorem \ref{T:nabla-i-complete} and Theorem \ref{thm:e2 left-invariant} it follows that metrics of type \eqref{eq:e2 i-complete metric} are $\mathcal{P}$-complete.
\end{theorem}

Standard computation shows that the covariant derivative of the Levi-Civita connection of these metrics vanishes, hence these spaces are locally symmetric. Therefore, because of \cite{Sz05} they are also $\mathcal P$-complete. The above theorem asserts that only the locally symmetric metrics are $\mathcal P$-complete on $E(2)$ and gives a new proof to the $P$-completeness of these particular metrics.

\section{The group E(1,1)}

Let $E(1,1)$ denote the group of isometries of the Minkowski plane, which has the following faithful representation.

\begin{equation}
    E(1,1)=\left\{ \begin{pmatrix}
    ch(t) & sh(t) & a \\
    sh(t) & ch(t) & b \\
    0 & 0 & 1
\end{pmatrix}\in GL(3,\bR)\Bigg|a,b,t\in\bR\right\}.
\end{equation}

Denote $[\mathfrak{e}(1,1),\mathfrak{e}(1,1)]$ by $V$. Let $e_1,e_2$ be a basis of $V$ and $e_0\in\mathfrak{e}(2)\setminus V$ such that $[e_0,e_1]=e_2$ and $[e_0,e_2]=e_1$. Let us define an indefinite bilinear form $(.,.)$ on $V$ such that

\begin{equation}
    (x,y)=x_1y_1-x_2y_2,
\end{equation}

where $x=x_1e_1+x_2e_2$ and $y=y_1e_1+y_2e_2$.

\begin{theorem}\label{thm:e11 left-invariant}
     Let $q$ be a left-invariant pseudo-Riemannian metric on $E(1,1)$. The Euler-Arnold vector field is $i\bR$-complete iff it belongs to one of the following:
     \begin{equation}\begin{split}\label{eq:e11 i-complete metric}
         &\lambda_0(e^*_0)^2+\lambda((e^*_1)^2-(e_2^*)^2) \\
         &\lambda_0(e_0^*)^2+(a/2+\lambda)(e_1^*)^2+(a/2-\lambda)(e_2^*)^2-a(e_1^*)(e_2^*),\text{ where }\frac{a}{\lambda_0}\geq 0\\
         &e_0^*e_1^*+\lambda (e_2^*)^2
     \end{split}\end{equation}
\end{theorem}

It is important to mention that, as a long but interesting calculation shows, among the above mentioned families of $\mathcal P$-complete metrics, the first is locally symmetric, but the others are not. Therefore, because of \cite{Sz05} they were already known to be $\mathcal P$-complete. However, not only do we present a new proof for the $\mathcal P$-completeness of these metrics, but we also present non-locally-symmetric $\mathcal P$-complete metrics.

\begin{proof}
    Similarly to the case of the previous group, \cite{G96} identifies several cases depending on the type of $q$. 

    If $q|_V$ is nondegenerate, then there exists a $\phi\in End(V)$ such that 

\begin{equation}
    q(x,y)=(\phi(x),y),
\end{equation}

for $x,y\in V$.

\textit{Case 1}. The matrix of $\phi$ is diagonalizable with eigenvalues $\lambda_1,\lambda_2$. There exists a basis $\{e_1,e_2,e_3\}$, such that for a trajectory $x(t)=x_1(t)e_1+x_2(t)e_2+x_3(t)e_3$, the Euler-Arnold equations are

\begin{equation}\begin{split}\label{eq: diag EA}
    \dot{x}_0(t)&=\frac{\lambda_2-\lambda_1}{\lambda_0}x_1(t)x_2(t) \\
    \dot{x}_1(t)&=-\frac{\lambda_2}{\lambda_1}x_0(t)x_2(t) \\
    \dot{x}_2(t)&=-\frac{\lambda_1}{\lambda_2}x_0(t)x_1(t).
\end{split}\end{equation}

Now, if $\lambda_1=\lambda_2$, then $x_0(t)$ is constant, and the metric is holomorphically complete. Otherwise we have first integrals

\begin{equation}\begin{split}\label{eq: diag first integrals}
    e&=\lambda_0x_0^2(t)+\lambda_1x_1^2(t)-\lambda_2x_2^2(t) \\
    m&=\lambda_1^2x_1^2(t)-\lambda_2^2x_2^2(t),
\end{split}\end{equation}

where $\lambda_0=q(e_0,e_0)$. Our method of proof is similar to the case of $E(2)$ we show that for certain initial conditions the corresponding trajectories are not $i\bR$-complete.

From the two first integrals, one can express $x_1(t)$ and $x_2(t)$ as functions in terms of each other such that

\begin{equation}\label{eq: diag x1 x2}
    x_1^2(t)=\frac{m+\lambda_2^2x_2^2(t)}{\lambda_1^2},\quad \text{ and}\quad x_2^2(t)=\frac{-m+\lambda_1^2x_1^2(t)}{\lambda_2^2}.
\end{equation}

Therefore by combining equation \eqref{eq: diag first integrals} and \eqref{eq: diag x1 x2} we get that

\begin{equation}
    e=\lambda_0x_0^2(t)+\frac{m}{\lambda_1}+\frac{\lambda_2(\lambda_2-\lambda_1)}{\lambda_1}x_2^2(t)=
\lambda_0x_0^2(t)+\frac{m}{\lambda_2}+\frac{\lambda_1(\lambda_2-\lambda_1)}{\lambda_2}x_1^2(t).
\end{equation}

By expressing $x_1^2(t)$ and $x_2^2(t)$ we obtain

\begin{equation}\label{eq: diag xi x0}
\begin{split}
    x_2^2(t)&=\left(e-\frac{m}{\lambda_1}-\lambda_0x_0^2(t)\right)\frac{\lambda_1}{\lambda_2(\lambda_2-\lambda_1)},\\
    &\text{and} \\
    x_1^2(t)&=\left(e-\frac{m}{\lambda_2}-\lambda_0x_0^2(t)\right)\frac{\lambda_2}{\lambda_1(\lambda_2-\lambda_1)}.
\end{split}    
\end{equation}

Combining equation \eqref{eq: diag xi x0} and \eqref{eq: diag EA} we get the following ODE for $x_0(t)$:

\begin{equation}
    \dot{x}_0^2(t)=\left(\frac{e}{\lambda_0}-\frac{m}{\lambda_0\lambda_1}-x_0^2(t)\right)\left(\frac{e}{\lambda_0}-\frac{m}{\lambda_0\lambda_2}-x_0^2(t)\right)
\end{equation}

With the choice of $x_1(0)=\lambda_2$ and $x_2(0)=\lambda_1$, $m=0$ using the continuity of $\dot x_0$ the either

\begin{equation}
    \dot{x}_0(t)=\frac{e}{\lambda_0}-x_0^2(t)\quad\text{or}\quad\dot{x}_0(t)=-\left(\frac{e}{\lambda_0}-x_0^2(t)\right)
\end{equation}

If $\frac{e}{\lambda_0}>0 $ and $x_0(0)=0$, then the corresponding trajectory is not $i\bR$-complete. With these choices of initial conditions, $\frac{e}{\lambda_0}$ becomes

\begin{equation}
    \frac{e}{\lambda_0}=x_0^2(0)+\frac{\lambda_1}{\lambda_0}\lambda_2^2-\frac{\lambda_2}{\lambda_0}\lambda_1^2=\frac{\lambda_1\lambda_2}{\lambda_0}(\lambda_2-\lambda_1).
\end{equation}

If $\frac{\lambda_1(\lambda_1-\lambda_2)}{\lambda_0\lambda_2}>0$, this gives a not $i\bR$-complete trajectory by Proposition \ref{P:t^2+a}. Otherwise, we can combine equation \eqref{eq: diag x1 x2}, \eqref{eq: diag xi x0} and \eqref{eq: diag EA} we arrive at the following ODE for $x_1(t)$

\begin{equation}\label{eq: diag x1}
    \dot{x}_1^2(t)=\frac{\lambda_2^2}{\lambda_1^2}\left(\frac{-m+\lambda_1^2x_1^2(t)}{\lambda_2^2}\right)\left(\frac{\lambda_2e-m}{\lambda_0\lambda_2}+\frac{\lambda_1(\lambda_1-\lambda_2)}{\lambda_0\lambda_2}x_1^2(t)\right)
\end{equation}

If $\frac{\lambda_1(\lambda_1-\lambda_2)}{\lambda_0\lambda_2}<0$, the trajectory belonging to $x_0(0)=x_1(0)=0\neq x_2(0)$ has $\lambda_2e-m=0$, so equation \eqref{eq: diag x1} becomes

\begin{equation}
    \dot{x}_1^2(t)=\left(-m\frac{\lambda_2^2}{\lambda_1^2}+x_1^2(t)\right)\frac{\lambda_1(\lambda_1-\lambda_2)}{\lambda_0\lambda_2}x_1^2(t).
\end{equation}
Using Theorem \ref{thm:x'^2}, the corresponding trajectory is not $i\bR$-complete.

\textit{Case 2}. The matrix of $\phi$ is nondegenerate and nondiagonalizable. Due to \cite{G96} there exists a basis $\{e_1,e_2,e_3\}$, such that for a trajectory $x(t)=x_1(t)e_1+x_2(t)e_2+x_3(t)e_3$, the Euler-Arnold equations become

\begin{equation}\begin{split}\label{eq: nondiag EA}
    \dot{x}_0(t)&=\frac{a}{2\lambda_0}(x_1(t)-x_2(t))^2 \\
    \dot{x}_1(t)&=-\frac{a}{\lambda}x_0(t)x_1(t)+\left(\frac{a}{\lambda}-1\right)x_0(t)x_2(t) \\
    \dot{x}_2(t)&=-\left(\frac{a}{\lambda}+1\right)x_0(t)x_1(t)+\frac{a}{\lambda}x_0(t)x_2(t) 
\end{split}\end{equation}

with first integrals

\begin{equation}\begin{split}
    e&=\lambda_0x_0^2(t)+(a/2+\lambda)x_1^2(t)+(a/2-\lambda)x_2^2(t)-ax_1x_2(t) \\
    m&=\lambda_0x_0^2(t)+\lambda/2(x_1^2(t)-x_2^2(t)),
\end{split}\end{equation}

Notice that if we substitute $m$ into the definition of $e$, we obtain the following identity:

\begin{equation}\label{eq: nondiag e m}
    e=-\lambda_0x_0^2(t)+2m+\frac{a}{2}(x_1(t)-x_2(t))^2.
\end{equation}

Combining equation \eqref{eq: nondiag EA} and \eqref{eq: nondiag e m} we get

\begin{equation}
    \dot{x}_0(t)=\frac{e-2m}{\lambda_0}+x_0^2(t).
\end{equation}

If $\frac{e-2m}{\lambda_0}<0$, $x_0(0)=0$, then using Proposition \ref{P:t^2+a} the corresponding trajectory is not $i\bR$-complete.

\begin{equation}
    \frac{e-2m}{\lambda_0}=\frac{a}{2\lambda_0}(x_1(0)-x_2(0))^2,
\end{equation}

which is negative iff $\frac{a}{2\lambda_0}<0$. In this case, the trajectory belonging to the initial condition $x(0)=0$, $x_1(0)\neq x_2(0)$ is not $i\bR$-complete.

If $\frac{a}{2\lambda_0}>0$ then $x_0(t)$ is $i\bR$-complete. Therefore we can use Proposition \ref{P:l-L-Omega} for the Euler-Arnold equation with the setting $\langle e_0\rangle=W$, $V=V$ and conclude that the connection in this case is $i\bR$-complete.

\textit{Case 3}. $\phi$ has complex eigenvalues. Then there exists a basis $\{e_1,e_2,e_3\}$, such that for a trajectory $x(t)=x_1(t)e_1+x_2(t)e_2+x_3(t)e_3$, Euler-Arnold equations are

\begin{equation}\begin{split}\label{eq: complex EA}
    \dot{x}_0(t)&=-\frac{\beta}{\lambda_0}(x_1^2(t)+x_2^2(t)) \\
    \dot{x}_1(t)&=x_0(t)(ax_1(t)+bx_2(t)) \\
    \dot{x}_2(t)&=x_0(t)(bx_1(t)-ax_2(t)),
\end{split}\end{equation}

where $a=\frac{2\alpha\beta}{\alpha^2+\beta^2},$ and $b=\frac{\beta^2-\alpha^2}{\alpha^2+\beta^2}$, with $\alpha, \beta\in\bR$ and $\beta\neq 0$. This system of equations has two first integrals:

\begin{equation}\begin{split}\label{eq: complex e m}
    e&=\lambda_0 x_0^2(t)+\alpha(x_1^2(t)-x_2^2(t))+2\beta x_1x_2 \\
    m&=b\lambda_0 x_0^2(t)+2\beta x_1x_2(t).
\end{split}\end{equation}

First, suppose that $\alpha\neq 0$. Then, by combining equation \eqref{eq: complex e m} and \eqref{eq: complex EA} we obtain the following ODE for $x_0(t)$:

\begin{equation}\begin{split}\label{eq:e(1,1) c-ev x_0}
    \dot{x}_0^2(t)&=\frac{\beta^2}{\lambda_0^2}\left((x_1^2(t)-x_2^2(t))^2+4x_1^2(t)x_2^2(t)\right) \\
    &=\frac{\beta^2}{\lambda_0^2}\left(\left(\frac{e-\lambda_0x_0^2(t)-(m-b\lambda_0x_0^2(t))}{\alpha}\right)^2+4\left(\frac{m-b\lambda_0x_0^2(t)}{2\beta}\right)^2\right).
\end{split}\end{equation}

If $(b-1)m=b(m-e)$, then $\dot{x}_0^2(t)$ is a complete square and \eqref{eq:e(1,1) c-ev x_0} becomes a Riccati-type equation. If we look at a trajectory for which $m=be$ then 

\begin{equation}
    (\dot{x}_0(t)-(x_0^2(t)-\frac{e}{\lambda_0}))(\dot{x}_0(t)+(x_0^2(t)-\frac{e}{\lambda_0}))=0.
\end{equation}

By the continuity of $\dot x_0$, either

\begin{equation}
    \dot{x}_0(t)=x_0^2(t)-\frac{e}{\lambda_0}\quad\text{ or}\quad\dot{x}_0(t)=-(x_0^2(t)-\frac{e}{\lambda_0})
\end{equation}
Using Proposition \ref{P:t^2+a} $x_0$ is not $i\bR$-complete if $\frac{e}{\lambda_0}>0$ and $x_0(0)=0$.

To find a trajectory such that $m=be$ we use their definition \eqref{eq: complex e m} and divide both sides by $\frac{\alpha}{\alpha^2+\beta^2}$, then after simplification we obtain the following:

\begin{equation}\label{eq: m=be}
    4\alpha\beta x_1(t)x_2(t)=(\beta^2-\alpha^2)(x_1^2(t)-x_2^2(t)).    
\end{equation}

If $\frac{\beta}{\lambda_0}>0$, then we can choose $x_1(0)=\alpha+\beta$ and $x_2(0)=\beta-\alpha$. This satisfies equation \eqref{eq: m=be}, so we only need to prove that $\frac e \lambda_0>0$.

\begin{equation}
    \frac{e}{\lambda_0}=\frac{1}{\lambda_0}(4\alpha^2\beta+2\beta^3-2\alpha^2\beta)=2\frac{\beta}{\lambda_0}(\alpha^2+\beta^2),
\end{equation}

so $\frac{e}{\lambda_0}>0$ and $x_0(t)$ is not $i\bR$-complete.

 If $\frac{\beta}{\lambda_0}<0$, then we can choose  $x_1(0)=\alpha-\beta$, $x_2(0)=\alpha+\beta$. Then

\begin{equation}
    \frac{e}{\lambda_0}=\frac{1}{\lambda_0}(-4\alpha^2\beta+2\beta(\alpha^2-\beta^2))=-2\frac{\beta}{\lambda_0}(\alpha^2+\beta^2).
\end{equation}

so $\frac{e}{\lambda_0}>0$ and $x_0(t)$ is not $i\bR$-complete.

Now, let $\alpha=0$. In this case, $e=m$, so we define a new first integral. Notice that
\begin{equation}
    c=\frac{e-m}{\alpha}=x_1^2(t)-x_2^2(t)+\frac{1-b}{\alpha}\lambda_0x_0^2(t)=x_1^2(t)-x_2^2(t)+\frac{2\alpha}{\alpha^2+\beta^2}\lambda_0x_0^2(t)
\end{equation} is a first integral, which is well-defined even for $\alpha=0$.

\begin{equation}\begin{split}\label{eq: complex e c}
    e&=\lambda_0 x_0^2(t)+2\beta x_1(t)x_2(t) \\
    c&=x_1^2(t)-x_2^2(t).
\end{split}\end{equation}

By combining equation $\eqref{eq: complex e c}$ and $\eqref{eq: complex EA}$ we obtain the following ODE for $x_0$:

\begin{equation}\label{eq: alpha0 x0}
    \dot{x}_0^2(t)=\frac{\beta^2}{\lambda_0^2}\left(c^2+4\left(\frac{e-\lambda_0x_0^2(t)}{2\beta}\right)^2\right).
\end{equation}

Now, if we look at a trajectory with initial condition $c=0$, equation \eqref{eq: alpha0 x0} simplifies to:

\begin{equation}
    \dot{x}_0(t)=\pm\left(x_0^2(t)-\frac{e}{\lambda_0}\right)
\end{equation}

If $x_0(0)=0$ and $e/\lambda_0>0$, then by Proposition \ref{P:t^2+a} the corresponding trajectory is not $i\bR$-complete. This means that we need to construct initial conditions such that $c=0$ to prove that these metrics are not $i\bR$-complete. 

If $\beta/\lambda_0>0$, then with the initial condition $x_1(0)=x_2(0)=1$, we have

\begin{equation}
    \frac{e}{\lambda_0}=2\frac{\beta}{\lambda_0}.
\end{equation}

Otherwise, if $\beta/\lambda_0<0$, then we can choose $x_1(0)=-x_2(0)=1$ to have

\begin{equation}
    \frac{e}{\lambda_0}=-2\frac{\beta}{\lambda_0}.
\end{equation}

In both cases, we have not $i\bR$-complete trajectories so neither of the corresponding metric is $i\bR$-complete.

\textit{Case 4}. $\phi$  is degenerate. Then there exists a basis such that the Euler-Arnold equations are

\begin{equation}\begin{split}\label{eq: deg EA}
    \dot{x}_0(t)&=2\lambda x_0(t)x_2(t) \\
    \dot{x}_1(t)&=-x_2(t)(x_0(t)+2\lambda x_1(t)) \\
    \dot{x}_2(t)&=\frac{x_0^2(t)}{2\lambda}
\end{split}\end{equation}

with first integrals

\begin{equation}\begin{split}\label{eq: deg first integral}
    e=&x_0(t)x_1(t)+\lambda x_2^2(t) \\
    m=&x_0^2(t)-4\lambda^2x_2^2(t).
\end{split}\end{equation}

By substituting $m$ into $\dot{x}_2(t)$ we get

\begin{equation}
    \dot{x}_2(t)=\frac{m+4\lambda^2x_2^2(t)}{2\lambda}.
\end{equation}

By using Proposition \ref{P:t^2+a} $x_2$ is not $i\bR$-complete iff $m<0$ and $x_2(0)=0$, but if $x_2(0)=0$, then $m=x_0^2\geq 0$. Therefore $x_2$ is always $i\bR$-complete, hence by applying Propositon \ref{P:l-L-Omega} to equation \eqref{eq: deg EA} we see that all the trajectories are $i\bR$-complete.\end{proof}

\begin{theorem}
    From Theorem \ref{T:nabla-i-complete} and Theorem \ref{thm:e11 left-invariant} it follows that metrics of type \eqref{eq:e11 i-complete metric} are $\mathcal{P}$-complete.
\end{theorem}

\section{Biinvariant connections}

Let $G$ be a Lie group and $\nabla$ be a left-invariant connection on $TG$. As we have seen, the $\mathcal P$-completeness of $\nabla$ is determined by the $i\bR$-completeness of its Euler-Arnold vector field. In general, left-invariant connections are in bijection with homogeneous quadratic vector fields on $\mathfrak g$. Therefore, determining all $\mathcal P$-complete left-invariant connections is equivalent to determining the $i\bR$-completeness of all homogeneous quadratic vector fields in $\dim G$. Without further assumptions on $\nabla$, one does not use the structure of the underlying group $G$. This is why up to this point we treated metric connections, because then the Euler-Arnold vector field was $\dot x=ad^*_x x$.

Another set of special cases of left-invariant connections are biinvariant connections. To see how they are related to the Lie group structure, we use the following theorem.

\begin{theorem}
    Let $\nabla$ be a left-invariant Koszul connection on a connected real Lie group $G$. Then $\nabla$ is biinvariant iff

    \begin{equation}\label{eq:biinv conn}
        [x,\nabla_yz]=\nabla_{[x,y]}z+\nabla_y[x,z],
    \end{equation}

    for all $x,y,z\in\mathfrak{g}$.\cite{PP19}
\end{theorem}

\begin{theorem}
    Let $G$ be a Lie group and $\mathfrak{g}$ its Lie algebra such that $[\mathfrak{g},\mathfrak{g}]\neq\mathfrak{g}$. Let $l\in\mathfrak{g}^*$ and $[\mathfrak{g},\mathfrak{g}]\subset Ker(l)$. Then the left-invariant connection $\nabla$ defined as 

    \begin{equation}\label{eq:l conn}
        \nabla_xy=\frac{1}{2}[x,y]+l(x)y+l(y)x
    \end{equation}

    is always torsion-free, biinvariant and $\mathcal{P}$-complete.
\end{theorem}

\begin{proof}
    The torsion-freeness of the given family of connections is trivial.
    To prove biinvariance we need to check that equation \eqref{eq:biinv conn} holds for every connection of the type \eqref{eq:l conn}
    \begin{equation}
    \begin{split}
        [x,\nabla_yz]=\frac{1}{2}[x,[y,z]]+l(y)[x,z]+l(z)[x,y] \\
        \nabla_{[x,y]}z=\frac{1}{2}[[x,y],z]+l([x,y])z+l(z)[x,y] \\
        \nabla_y[x,z]=\frac{1}{2}[y,[x,z]]+l(y)[x,z]+l([x,z])y.
    \end{split}        
    \end{equation}

    Using Theorem \ref{T:nabla-i-complete} The $\mathcal{P}$-completeness of the connection follows from the $i\bR$-completeness of the corresponding Euler-Arnold vector field. The Euler-Arnold vector field is
    
    \begin{equation}
        \dot x(t)=-2l(x(t))x(t),
    \end{equation}

    which by Proposition \ref{P:l-L-Omega} is always $i\bR$-complete.
\end{proof}

\section{Solvable, non-unimodular Lie groups in dimension 3}

John Milnor \cite{M76} proved that the 3-dimensional non-unimodular Lie algebras can be characterized as follows: 

Consider the linear functional $tr(ad_x):\mathfrak{g}\rightarrow \bR$. Since the algebra is non-unimodular, this is not identically zero, hence has a 2-dimensional kernel, call it the \textit{unimodular kernel} denoted by $V$. Since this Lie-subalgebra is unimodular and 2-dimensional, it is commutative.

There exists an $e_0\in\mathfrak{g}\setminus V$ such that $tr(ad_{e_0})=2$. Now, either $A:=ad_{e_0}|_{V}$ acts on $V$ as the identity (thus $\mathfrak{g}$ belongs to $\mathcal{F}$) or there is $e_1\in V$ which is not an eigenvector of $A$. Thus for

\begin{equation}
    e_2=[e_0,e_1]
\end{equation}

we have the following.

\begin{equation}\begin{split}
    [e_0,e_1]&=e_2 \\
    [e_0,e_2]&=-De_1+2e_2 \\
    [e_1,e_2]&=0,
\end{split}\end{equation}

where $D$ is the only degree of freedom we have, so it is a complete isomorphism invariant. Notice that $det(A)=D$. For $x=x_0e_0+x_1e_1+x_2e_2\in\mathfrak{g}$

\begin{equation}
    ad_x=\begin{pmatrix}
    0 & 0 & 0 \\
    Dx_2 & 0 & -Dx_0 \\
    -x_1-2x_2 & x_0 & 2x_0
\end{pmatrix}
\end{equation}

In the following theorem, we focus on the Lie groups belonging to the $D=0$ case. This Lie algebra is isomorphic to the Lie algebra of $\textit{Aff}(1)\times \bR$, so we emphasize this isomorphism. We introduce a new basis $f_0=\frac{1}{2}e_0$, $f_1=e_2$, $f_2=2e_1-e_2$. To point out the structure of the Lie algebra even more let us denote $x_0f_0+x_1f_1$ by $x_a$. It is easy to see that for these Lie algebra elements the following identities hold:

\begin{equation}\begin{split}
    [f_0,f_1]&=\frac{1}{2}[e_0,e_2]=e_2=f_1 \\
    [f_0,f_2]&=\frac{1}{2}[e_0,2e_1-e_2]=0 \\
    [f_1,f_2]&=[e_2,2e_1-e_2]=0 
\end{split}  
\end{equation}

In other words, for $x,y\in\mathfrak{g}$ and $x=x_0f_0+x_1f_1+x_2f_2$ we obtain

\begin{equation}\label{eq:lie D0}
    [x,y]=(x_0y_1-y_0x_1)f_1.
\end{equation}

\begin{theorem}\label{thm:D0}
    Let $G$ be a solvable, non-unimodular Lie group of dimension 3 with $D=0$ and $\nabla$ a biinvariant torsion-free connection on $TG$. Then $\nabla$ takes the form
    
    \begin{equation}
        \nabla_xy=\frac{1}{2}[x,y]+l(x)y_a+l(y)x_a+(\nabla_x y)_2 f_2.      
    \end{equation}

    where $l\in\mathfrak{g*}$, with $l(f_1)=0$ and $(\nabla_x y)_2$ does not depend on $x_1$ or $y_1$.
\end{theorem}

\begin{proof}

    To prove the statement of the theorem, we only need to determine the class of connections for wich equation \eqref{eq:biinv conn} holds. Suppose we have such a connection $\nabla$. Then by using equation \eqref{eq:lie D0} we obtain
    \begin{equation}\label{eq:biinv D0}
        \begin{split}
            [x,\nabla_y z]&=(x_0(\nabla_y z)_1-(\nabla_y z)_0x_1)f_1 \\
            \nabla_{[x,y]}z+\nabla_y[x,z]&=(x_0y_1-y_0x_1)\nabla_{f_1}z+(x_0z_1-z_0x_1)\nabla_y f_1.
        \end{split}
    \end{equation}

    for $x,y,z\in\mathfrak{g}$. If we choose $x=f_0$ in equation \eqref{eq:biinv D0} then we obtain the following identity

    \begin{equation}\label{eq:D0 x0}
        (\nabla_y z)_1f_1=y_1\nabla_{f_1}z+z_1\nabla_y f_1.
    \end{equation}

    Substituting $y=z=f_1$ we get

    \begin{equation}\label{eq:D0 111}
        (\nabla_{f_1}f_1)_1f_1=2\nabla_{f_1}f_1,
    \end{equation}

     which is therefore equal to $0$. Setting $y=f_1$ $z=f_j$ in equation \eqref{eq:D0 x0} for $j\neq 1$

    \begin{equation}\label{eq:D0 f1}
        (\nabla_{f_1}f_j)_1f_1=\nabla_{f_1}f_j,
    \end{equation}

    hence 
    
    \begin{equation}\label{eq:D0 l0}
        \nabla_{f_1}x=l_0(x)f_1,
    \end{equation} for some $l_0\in\mathfrak{g}^*$, with $l_0(f_1)=0$ because of the equation \eqref{eq:D0 111}. By combining equation \eqref{eq:D0 f1} and torsion-freeness we also get that.
    
    \begin{equation}\label{eq:D0 f1*}
        \nabla_z f_1=[z,f_1]+\nabla_{f_1}z=z_0f_1+l_0(z)f_1
    \end{equation}

    Equations \eqref{eq:D0 l0} and \eqref{eq:D0 f1*} imply that $(\nabla_x y)_0$ and $(\nabla_x y)_2$ do not depend on $x_1$ or $y_1$.
    
    Putting together equations \eqref{eq:D0 x0}, \eqref{eq:D0 l0} and \eqref{eq:D0 f1*} we can see that

    \begin{equation}\label{eq:D0 x1}
        (\nabla_y z)_1f_1=(l_0(y_1z+z_1y)+y_0z_1)f_1.
    \end{equation}

    Now, looking at equation \eqref{eq:biinv D0} at $x=f_1$ we get

    \begin{equation}\label{eq:D0 x1 f1}
        (\nabla_y z)_0f_1=y_0\nabla_{f_1}z+z_0\nabla_y f_1=l_0(y_0z+z_0y)f_1+z_0y_0f_1.
    \end{equation}

    Combining equation \eqref{eq:D0 x1} and \eqref{eq:D0 x1 f1} we get that

    \begin{equation}\label{eq:D0 nabla}
        \begin{split}
            \nabla_xy=&(\nabla_x y)_0f_0+(\nabla_x y)_1f_1+(\nabla_x y)_2 f_2 \\
            =&(l_0(x_0y+y_0x)+y_0x_0)f_0+\\&(l_0(x_1y+y_1x)+x_0y_1)f_1+(\nabla_x y)_2 f_2 \\
            =&\frac{1}{2}[x,y]+(l_0(x_0y+y_0x)+y_0x_0)f_0+\\&(l_0(x_1y+y_1x)+\frac{1}{2}x_0y_1+\frac{1}{2}y_0x_1)f_1+
            (\nabla_x y)_2 f_2
        \end{split}        
    \end{equation}

    Denote $\frac{1}{2}f_0^*+l_0\in\mathfrak{g}^*$ by $l$. Then equation \eqref{eq:D0 nabla} simplifies to

    \begin{equation}\label{eq:D0 nabla'}
        \nabla_xy=\frac{1}{2}[x,y]+l(x)y_a+l(y)x_a+(\nabla_x y)_2 f_2.
    \end{equation}

    We only need to prove at this point that for a connection of the form \eqref{eq:D0 nabla'} equation \eqref{eq:biinv conn} always holds.

    \begin{equation}
        \begin{split}
            [x,\nabla_y z]&=\frac{1}{2}[x,[y,z]]+l(y)[x,z]+l(z)[x,y] \\
            \nabla_{[x,y]}z&=\frac{1}{2}[[x,y],z]+l(z)[x,y]\\
            \nabla_y[x,z]&=\frac{1}{2}[y,[x,z]]+l(y)[x,z]
        \end{split}
    \end{equation}

    The sum of the latter two equations clearly equals the first one, thus a connection of the form \eqref{eq:D0 nabla'} is always biinvariant.\end{proof}

\begin{cor}\label{C:D0}
    Let $G$ be a solvable, non-unimodular Lie group of dimension three with $D=0$. Then there exist biinvariant torsion-free connections $\nabla$ on $TG$ which are not $i\bR$-complete.
\end{cor}

This theorem in itself is not of particular interest. The aim of presenting this statement as a theorem is to highlight the difference between the $D=0$ case and the $D\neq 0$ case. \ref{thm:general F biinv}

\begin{proof}
    Let $G$ be a solvable, non-unimodular Lie group of dimension three with $D\neq 0$ and $\nabla$ a biinvariant torsion-free Koszul connection on $TG$. Then by Theorem \ref{thm:D0} $\nabla$ takes the form
    
    \begin{equation}
        \nabla_xy=\frac{1}{2}[x,y]+l(x)y_a+l(y)x_a+(\nabla_x y)_2 f_2.      
    \end{equation}

    for some $l\in\mathfrak{g*}$, with $l(f_1)=0$ and $(\nabla_x y)_2$ not depending on $x_1$ or $y_1$. 

    Let $l=0$, and $(\nabla_x y)_2=-\frac{c_1}{2}(x_0y_2+x_2y_0)-c_2x_2y_2$, with $c_2<0$.  Then the Euler-Arnold equation is

    \begin{equation}
        \begin{split}
            \dot x_0(t)&=0\\
            \dot x_1(t)&=0\\
            \dot x_2(t)&=c_1x_0x_2+c_2x_2^2.
        \end{split}
    \end{equation}

    This system of ODEs is not $i\bR$-complete because of Proposition \ref{P:t^2+a}.
\end{proof}

\section{General \texorpdfstring{$\mathcal F$}-type Lie groups}

In our previous paper \cite{FSz26} we investigated with Róbert Szőke the biinvariant connections of $\mathcal F$-type Lie groups in the context of $\mathcal P$-completeness. After that, we tried to do the same for $E(2)$ and $E(1,1)$. In the process we noticed that biinvariant connections behave in a similar way for these groups. Therefore I came up with a common generalization of these groups for which the same proof works.

\begin{definition}
     Let $G$ be a Lie group with Lie algebra $\mathfrak{g}$. Let  $\mathfrak{g}=\langle e_0\rangle\oplus V$, where $V$ is abelian, and $ad_{e_0}|_V=A\in GL(V)$. Then we call $G$ $general \ \mathcal{F}-type$. Denote a generic element $x$ of $\mathfrak{g}$ by $x=x_0e_0+x_v$, where $x_v\in V$.
\end{definition}

Notice that

\begin{equation}\label{eq:lie e11}
    [x,y]=x_0Ay_v-y_0Ax_v.
\end{equation}

From equation \eqref{eq:lie e11} it follows that groups of general $\mathcal F$-type are 2-step solvable.

\begin{example}
    Let 
    \begin{equation}\label{eq:general F}
        G=\left\{ \begin{pmatrix}
    
     e^{At} & v \\
     0 & 1
\end{pmatrix}\in GL(n+1,\bR)\Bigg|t\in\bR,v\in\bR^n,A\in GL(n,\bR)\right\},
    \end{equation}
    then $G$ is general $\mathcal{F}$-type.
\end{example}

It might seem strange at first that we look at $e^{A}$, with $A\in GL(n,\bR)$ and not $End(n,\bR)$. This might appear to be a typo, but Corollary \ref{C:D0} and Theorem \ref{thm:general F biinv} prove that the degeneracy of $A$ leads to vastly different behaviors.

\begin{example}
    Let $A=I_n$, then $G$ in \eqref{eq:general F} is $\mathcal{F}$-type.
\end{example}

\begin{example}
    Let $n=2$ and 
    \begin{equation}
        A=\begin{pmatrix}
            0 &-1 \\
            1 & 0
        \end{pmatrix}
    \end{equation}then $G$ in \eqref{eq:general F} is $E(2)$.
\end{example}

\begin{example}
    Let $n=2$ and 
    \begin{equation}
        A=\begin{pmatrix}
            0 & 1 \\
            1 & 0
        \end{pmatrix}
    \end{equation}then $G$ in \eqref{eq:general F} is $E(1,1)$.
\end{example}

\begin{example}
    Let $n=2$ and 
    \begin{equation}
        A=\begin{pmatrix}
            0 &-D \\
            1 & 2
        \end{pmatrix}
    \end{equation}then $G$ in \eqref{eq:general F} is a solvable, non-unidular Lie group of dimension 3.
\end{example}
\begin{theorem}\label{thm:biinv metric}
    Let $q$ be a left-invariant pseudo-Riemannian metric on a connected real Lie group $G$. Then $q$ is biinvariant iff

    \begin{equation}\label{eq:biinv metric}
        q([x,y],z)+q(y,[x,z])=0,
    \end{equation}

    for all $x,y,z\in\mathfrak{g}$.\cite{O83}
\end{theorem}

Lie groups for which there exist a biinvariant Riemannian metric have been characterized before by Milnor in \cite{M76}. To our knowledge, a similar classification does not exist for pseudo-Riemann metrics.

\begin{theorem}
    Let $G$ be a general $\mathcal{F}$-type Lie group. Then there exists no biinvariant pseudo-Riemannian metric on $G$.
\end{theorem}

\begin{proof} 
Let $G$ be a general $\mathcal{F}$-type Lie group and suppose that $q$ is a biinvariant pseudo-Riemannian metric on $G$. By Theorem \ref{thm:biinv metric}, we only need to check that Equation \eqref{eq:biinv metric} holds. If it holds, then in particular

\begin{equation}
    q([e_0,x],y)+q(e_0,[x,y])=0
\end{equation}

for all $x,y\in V$. But since $V$ is abelian, it means that $[e,x]\in V\cap V^{\perp}$ but since $ad_{e_0}$ is surjective $V=V^{\perp}$. This can only happen, if $\dim(V)=1$, since $V$ has codimension 1.

If $\dim(V)=1$, then $G$ is isomorphic to $\textit{Aff}(1)$, which has no biinvariant pseudo-Riemannian metric, as is well-known. \cite{V19}
\end{proof}

The classification of pseudo-Riemannian metrics on general $\mathcal F$-type Lie groups in terms of $\mathcal P$-completeness is beyond the scope of this article. Pseudo-Riemannian metrics for non-unimodular 3 dimensional metrics have been classified up to automorphism by Ha and Lee. \cite{HL09}\cite{HL23}

\begin{theorem}\label{thm:general F biinv}
    Let $G$ be a general $\mathcal{F}$-type Lie group. Let $\nabla$ be a torsion-free biinvariant Koszul connection on $TG$. Then there exists an $a\in\bR$ such that

    \begin{equation}\label{eq:general biinv}
        \nabla_x y=\frac{1}{2}[x,y]+\frac{a}{2}(x_0y+y_0x).
    \end{equation}
    for all $x,y\in\mathfrak g$. These connections are all $\mathcal P$-complete.

\end{theorem}

\begin{proof}

    The case of $A=I$ was treated in \cite[Theorem 9.1.]{FSz26}. The fact that

    \begin{equation}\label{eq:F-type biinv}
        x_0^2\nabla_{e_0}e_0+x_0(\nabla_{e_0}x_v+\nabla_{x_v}e_0)=x_0(\nabla_{e_0}e_0)_0x
    \end{equation}

    and $(\nabla_{x_v}x_v)_0=0$ can be proven in a similar fashion.

      To treat $(\nabla_{x_v}x_v)_v$ we choose $x=e_0$, and use equation \eqref{eq:lie e11} and \eqref{eq:biinv conn} to get

    \begin{equation}\label{eq:e11 biinv V}
        A(\nabla_y z)_v=\nabla_{Ay_v} z+\nabla_y Az_v.
    \end{equation}

    Using equation \eqref{eq:F-type biinv}

    \begin{equation}\label{eq:general biinv beta}
            \nabla_x x           =x_0(\nabla_{e_0}e_0)_0x+\nabla_{x_v}x_v.     
    \end{equation}

    Since $\nabla$ is torsion-free, the values of $\nabla_xx$ completely determine it completely.

    We also need to show, that the left-invariant connection defined by equation \eqref{eq:general biinv beta} is in fact biinvariant. Denote $(\nabla_{e_0}e_0)_0$ by $a$. To check biinvariance, it suffices to prove that $\nabla$ satisfies equation \eqref{eq:biinv conn}.

    \begin{equation}
        \begin{split}
            [x,\nabla_yz]&=[x,\frac{1}{2}[y,z]+\frac{a}{2}(y_0z+z_0y)+\nabla_{y_v}z_v]\\
            &=\frac{1}{2}[x,[y,z]]+\frac{a}{2}y_0[x,z]+\frac{a}{2}z_0[x,y]+x_0A(\nabla_{y_v}z_v)\\
            \nabla_{[x,y]}z&=\frac{1}{2}[[x,y],z]+\frac{a}{2}z_0[x,y]+\nabla_{x_0Ay_v-y_0Ax_v}z_v\\
            \nabla_y[x,z]&=\frac{1}{2}[y,[x,z]]+\frac{a}{2}y_0[x,z]+\nabla_{y_v}(x_0Az_v-z_0Ax_v).
        \end{split}
    \end{equation}

    The sum of the latter two equations has to equal the first. From Equation \eqref{eq:e11 biinv V} we have that

    \begin{equation}
        \nabla_{y_0Ax_v}z_v+\nabla_{y_v}(z_0Ax_v)=0
    \end{equation}

    for all $x,y,z\in \mathfrak{g}$. With the choice of $y_0=0$, $z_0=1$ and $x_v=A^{-1}z_v$ we obtain

    \begin{equation}
        \nabla_{y_v}z_v=0.
    \end{equation}
    
    And thus, the Euler-Arnold equation is 

    \begin{equation}
        \dot{x}(t)=-a x_0(t)x(t).
    \end{equation}
    
    Proposition~\ref{P:l-L-Omega} implies that such a vector field is always $i\bR$-complete, and by Theorem \ref{T:nabla-i-complete}   the connection is $\mathcal P$-complete.
\end{proof}

\begin{theorem}\label{thm: general F covariant}
    Let $G$ be a general $\mathcal{F}$-type Lie group and $\nabla$ be a torsion-free biinvariant Koszul connection on $TG$ defined by Equation \eqref{eq:general biinv}. Then 

    \begin{equation}
        R(x,y)z=\frac{1}{4}z_0(A^2-a^2I)\left(x_0y-y_0x\right),
    \end{equation}

    and 

    \begin{equation}
        \nabla_w(R)(x,y,z)=\frac{a}{2}w_0z_0(A^2-a^2I)(x_0y-y_0x).
    \end{equation}
\end{theorem}

\begin{proof}
    Similarly to the previous paper. \cite[Proposition 9.2.]{FSz26}
\end{proof}

It follows that for a generic $a\in\bR$ ($a\neq0$ and $a^2I\neq A^2$) biinvariant connections on general $\mathcal F$-type Lie groups will not be locally symmetric.

\section{Conclusion}

In this paper, we present several new examples for $\mathcal{P}$-complete but not locally symmetric Koszul manifolds. In Theorem \ref{thm:e11 left-invariant} we classified left-invariant pseudo-Riemann metics on $E(1,1)$ in terms of $\mathcal{P}$-completeness. In Theorem \ref{thm:general F biinv} we showed that all biinvariant connections on general $\mathcal F$-type Lie groups are $\mathcal P$-complete, but according to Theorem \ref{thm: general F covariant} they are generically not locally symmetric.

\textbf{Declarations}

\textbf{Conflict of interest} There is no conflict of interest.

\textbf{Ethical approval} Not applicable.


\begin{thebibliography}{9}
\bibitem[1]{A66}
Vladimir I.  Arnold: 
\emph{Sur la géométrie différentielle des groupes de Lie de dimension infinie et ses applications á l'hydrodynamique des fluides parfais}, 
Annales de l'intitut Fourier, tome 16, n°1, 319-361., 1966

\bibitem[2]{G96}
Mohammed Guediri:
\emph{On completeness of left-invariant Lorentz metrics on solvable Lie groups}, Rev. Mat. Univ. Complut. Madrid 9, no. 2, 337-350, 1996
\bibitem[3]{FSz26}
Balázs Forman, Róbert Szőke:
\emph{Polarization complete left-invariant connections} arxiv/2609.15673, 2026

\bibitem[4]{GKKL15}
M. Glas, P. Konstantis, A. Krause and F. Loose:
\emph{Bianchi's classification of 3-dimensional Lie algebras revisited} Journal of Lie Theory 25, 975–1003., 2009


\bibitem[5]{HL09}Ku Yong Ha and Jong Bum Lee 
\emph{Left Invariant Metrics and Curvatures on Connected, Simply Connected
three-Dimensional Lie Groups}
Mathematische Nachrichten. 282. 868 - 898., 2009
\bibitem[6]{HL23}Ku Yong Ha and Jong Bum Lee
\emph{Left Invariant Lorentzian Metrics and Curvatures on non-unimodular Lie groups of dimension three} J. Korean Math. Soc. 60 No. 1, pp. 143-163, 2023
\bibitem[7]{HI03}
Stephan Halverscheid and Andrea Iannuzzi:
\emph{Maximal complexifications of certain homogeneous Riemannian manifolds},
Trans. AMS, 355, No11, 4581-4594, 2003
\bibitem[8]{HI09}
Stephan Halverscheid, Andrea Iannuzzi:
\emph{A family of adapted complexifications for $SL_2(\bR)$},
Ann. Scuola Norm. Pisa Cl. Sci. (5) Vol. VIII (2009), 17-49, 2009

\bibitem[9]{KOTT16}
A. Kubo, K. Onda, Y. Taketomi, and H. Tamaru:
\emph{On the moduli spaces of left-invariant pseudo-Riemannian metrics on Lie groups},
Hiroshima Math. J., 46(3), 357-374, 2016
\bibitem[10]{IV07} J. Inoguchi, J. Van der Veke:
\emph{Parallel surfaces in the motion groups E(1, 1) and E(2)}, Bulletin of the Belgian Mathematical Society,  Simon Stevin 14, 321–332, 2007
\bibitem[11]{LSz91}
László Lempert and  Róbert Szőke:
\emph{Global solutions of the homogeneous, complex Monge-Amp\'ere equation and complex structures on the tangent bundle of Riemannian manifolds}, 290, 689-712, 1991
Math. Ann., 
\bibitem[12]{L19}
László Lempert:
 \emph{Adapted complex and involutive structures} arXiv:1901.03388, 2019
\bibitem[13]{LB80}
Claude LeBrun: 
\emph{Spaces of complex geodesics and related structures},
PhD thesis, Hertford College, Oxford, 1980
\bibitem[14]{LB83}
Claude LeBrun: 
\emph{Spaces of complex null geodesics in complex-Riemannian geometry},
Trans. AMS, Vol. 278, No 1, 209-231, 1983

\bibitem[15]{M76}
John Milnor: 
\emph{Curvatures of left invariant metrics on Lie groups},
Adv. in Math. 21, 293-329, 1976

\bibitem[16]{O83}
Barrett O'Neill:
\emph{Semi-Riemannian geometry},
Academic Press, 1983
\bibitem[17]{PP19}
P.G.Pripoae and C.L.Pripoae:
\emph{Invariant connections on Lie groups}
Balkan Journal of Geometry and Its Applications, Vol 24., No.1., pp.51-64., 2019


\bibitem[18]{Sz01}
Róbert Szőke:
\emph{Involutive structures on the tangent bundle of symmetric spaces},
Math. Ann., 319, 319-348, 2001
\bibitem[19]{Sz04}
Róbert Szőke:
\emph{Canonical complex structures associated to connections and complexifications of Lie groups},
Math. Ann. 329, 553-591, 2004
\bibitem[20]{Sz05}
Róbert Szőke:
\emph{Complex crowns of symmetric spaces},
Intern. J. Math., 16(8), 889-902, 2005
\bibitem[21]{V19}Fabricio Valencia:
\emph{Notes on flat pseudo-Riemannian manifolds} arXiv:1903.08940, 2019

\end{thebibliography}
\end{document}